\documentclass{amsart}

\usepackage{amsfonts,amssymb,amsmath,enumerate,verbatim,mathtools,tikz,csquotes}
\usetikzlibrary{matrix}
\usetikzlibrary{arrows}

\usepackage[all]{xy}
\SelectTips{cm}{}

\newcommand\ext{\operatorname{Ext}}

\newcommand\fm{{\mathfrak m}}
\newcommand\fn{{\mathfrak n}}

\newcommand\me{{\mathcal{E}}}
\newcommand\mk{{\mathcal{K}}}

\newcommand\mv{{\mathcal{V}}}

\newcommand\Hom{\operatorname{Hom}}

\newcommand{\ov}{\overline}

\newcommand{\im}{\operatorname{Im}}
\newcommand{\reg}{\operatorname{reg}}

\newcommand{\rank}{\operatorname{rank}}

\newcommand{\xra}{\xrightarrow}

\newtheorem{theorem}{Theorem}[section]

\newtheorem{proposition}[theorem]{Proposition}
\newtheorem{lemma}[theorem]{Lemma}
\newtheorem{corollary}[theorem]{Corollary}

\newtheorem{question}[theorem]{Question}

\theoremstyle{definition}

\newtheorem*{definition}{Definition}

\theoremstyle{remark}
\newtheorem{remark}[theorem]{Remark}

\newtheorem{notation}[theorem]{Notation}

\newtheorem{chunk}[theorem]{}

\numberwithin{equation}{theorem}

\begin{document}

\begin{abstract}
In this work, we prove that if a graded, commutative algebra $R$ over a field
$k$ is not Koszul then, denoting by $\fm$ the maximal homogeneous
ideal of $R$ and by $M$ a finitely generated graded $R$-module, the nonzero modules of the form $\fm M$ have infinite
Castelnuovo-Mumford regularity. We also prove that over complete
intersections which are not Koszul, a nonzero direct summand of a
syzygy of $k$ has infinite regularity. Finally, we relate the vanishing
of the graded deviations of $R$ to having a nonzero direct summand
of a syzygy of $k$ of finite regularity.
\end{abstract}

\title{Modules of infinite regularity over commutative graded rings}
\author{Luigi Ferraro}
\maketitle
\setcounter{section}{1}
\section*{Introduction}
Let $k[x_1,\ldots,x_e]$ be a polynomial ring with coefficients in a field $k$, graded by setting
$\deg x_i=d_i\geq1$, and throughout the paper $R$ will denote the ring $k[x_1,\ldots,x_e]/I$ where $I$ is a
homogeneous ideal. We denote by $\fm$ the maximal homogeneous ideal
of $R$. The size of a minimal free resolution of a graded $R$-module
$M$ is measured by its graded \emph{Betti numbers}
$\beta_{i,j}(M)=\rank_k\ext_R^{i,j}(M,k)$. Invariants arising from
the Betti numbers of $M$ are its \emph{projective dimension} and its
\emph{Castelnuovo-Mumford regularity}. In \cite{ae}, Avramov and
Eisenbud prove that the regularity of $k$ is finite if and only if
the regularity of every finitely generated module is finite. The main goal of this paper is to construct finitely generated modules of infinite regularity, provided that the regularity of $k$ is also infinite. As a result these modules can be used as \enquote{test modules} to test if $k$, and therefore all finitely generated modules, have finite regularity.

In \cite{extremal}, Avramov proves that if $(Q,\fn)$ is a local commutative ring and $M$ a tightly embeddable $Q$-module (see the beginning of Section 1 for the definition of tightly embeddable modules in the graded setting) then $M$ has infinite projective dimension, provided that $Q$ is not regular. Avramov achieves this result by proving an inequality of formal power series involving the Poincar\'{e} series of the residue field of $Q$ and the Poincar\'{e} series of $M$. Examples of tightly embeddable modules are modules of the form $\fn L$ for some finitely generated $Q$-module $L$.
In Section 1 of this article, we prove a graded version of this result. In Theorem \ref{main1} we prove an inequality of formal power series involving the bigraded Poincar\'{e} series of $k$ and the bigraded Poincar\'{e} series of a tightly embeddable module. As a corollary of this theorem we prove that tightly embeddable modules have infinite regularity, provided that the regularity of $k$ is also infinite. This also gives a way to test if the regularity of $k$ is finite using tightly embeddable modules. In particular if a nonzero $R$-module of the form $\fm L$, with $L$ a graded and finitely generated $R$-module, has finite regularity then so does $k$.

In Section 2, we prove that if $R$ is a complete intersection then there is an inequality involving the bigraded Poincar\'{e} series of $k$ and the bigraded Poincar\'{e} series of nonzero direct summands of syzygies of $k$, and more generally of modules satisfying a technical hypothesis. As a consequence we deduce that if $R$ is a complete intersection and the regularity of $k$ is infinite, then any nonzero summand of a syzygy of $k$ has infinite regularity. We ask whether or not these results hold true if one removes the complete intersection hypothesis. In an attempt to remove the complete intersection hypothesis, we study, in the third section of this paper, the relation between the vanishing of the graded deviations of $R$, denoted by $\varepsilon_{i,j}(R)$, and the existence of a homomorphic image of a syzygy of $k$ with finite regularity.

We prove, under a technical hypothesis, that if a homomorphic image of a syzygy of $k$ has finite regularity, then the graded deviations of $R$, $\varepsilon_{i,j}(R)$, are $0$ if $i\gg0$ and $i\neq j$. We raise the question whether or not if $\varepsilon_{i,j}(R)=0$ for $i\gg0$ and $i\neq j$, then $\varepsilon_{i,j}(R)=0$ for $i\geq3$ and $i\neq j$. In an attempt to answer this question we show that if $d$ is an even integer larger or equal to 4 and $\varepsilon_{i,j}(R)=0$ for $i\geq d$ and $i\neq j$ then $\varepsilon_{i,j}(R)=0$ for $i\geq d-1$ and $i\neq j$.

\setcounter{section}{0}
\section{Tightly embeddable modules}
In this section we construct a class of modules, that have infinite regularity provided that $k$ has infinite regularity.
\begin{definition}
Let $M$ be a graded $R$-module. We say that $M$ is \emph{tightly embeddable} if there exists a finitely generated
graded $R$-module $L$ such that
\[
\fm M\varsubsetneqq \fm L\subseteq M\subseteq L.
\]
In this case, $M\subseteq L$ is a \emph{tight embedding}.
\end{definition}
\begin{chunk}
It follows from Nakayama's Lemma that if $L$ is finitely generated, then $\fm^iL\subseteq\fm^{i-1}L$ is a tight embedding for each $i\geq 1$ such that $\fm^iL$ is not zero.
\end{chunk}
Motivated by the results in \cite{extremal}, we explore the relation between tightly embeddable modules and regularity.

We establish the convention that if $V$ is a graded $k$-vector space then $V^j=V_{-j}$:

\begin{chunk}
The Hilbert series of a graded $k$-vector space $V$ is
\[
H_V(s)=\sum_j\rank_k V_js^j.
\]
If $V$ is a bigraded $k$-vector space such that for any fixed $i$ there are only finitely many non zero $V_{i,j}$, then the Hilbert series of $V$ is
\[
H_V(s,t)=\sum_{i,j}\rank_k V_{i,j}s^jt^i\in\mathbb{Z}[s^{\pm 1}][[t]].
\]
\end{chunk}

\begin{chunk}
The bigraded Poincar\'{e} series of a finitely generated graded $R$-module $M$ is
\[
P_M^R(s,t)=\sum_{i,j}\beta_{i,j}^R(M)s^jt^i\in\mathbb{Z}[s^{\pm 1}][[t]],
\]
where $i$ is the homological degree and $j$ is the internal degree.
\end{chunk}

\begin{chunk} Let $\displaystyle \sum_ia_it^i$ and $\displaystyle \sum_ib_it^i$ be
formal power series in $\mathbb{Z}[[t]]$. If $a_i\leq b_i$ for every
$i$, then we write
\[
\sum_ia_it^i\preceq\sum_ib_it^i
\]
or
\[
\sum_ib_it^i\succeq\sum_ia_it^i.
\]
\end{chunk}
\begin{chunk}
If $A$ is a graded $k$-algebra and $X$ is a graded $A$-module, then the $s$th \emph{shift} of $X$ is the graded $A$-module $(\Sigma^s X)^i=X^{i-s}$. Here, the $A$-action is defined as $a\cdot x=(-1)^{|a|s} ax$, where $|a|$ denotes the degree of $a$. We denote $\Sigma^1 X$ by $\Sigma X$.
\end{chunk}
\begin{remark}
In accordance with a bigraded version of Gulliksen and Levin \cite{Levin}, the
algebra $\ext_R(k,k)$ is a (bi-)graded Hopf $k$-algebra. This
algebra is the universal enveloping algebra of a (bi-)graded Lie
algebra $\pi^{*,\_}(R)$ where the first degree is homological and the
second is the internal degree. This follows from bigraded versions
of the theorems in \cite[(5.18)]{MilnorMoore} (characteristic 0),
\cite[Theorem 17]{andre} (positive odd characteristic) \cite[Theorem
2]{sjodin} (characteristic 2).
\end{remark}

\begin{notation}
We denote by $\me$ the algebra $\ext_R(k,k)$. For every $R$-module
$M$, we denote by $\me(M)$ the left $\me$-module $\ext_R(M,k)$. If $X$ is an $R$-module, then we denote by $X^\vee$ the module $\Hom_R(X,k)$, we point out that if $X$ is annihilated by $\fm$ then $\Hom_R(X,k)=\Hom_k(X,k)$. These notations will be used throughout the paper.
\end{notation}

The following theorem is a graded version of \cite[Lemma 6]{extremal} and has a similar proof.

\begin{theorem} \label{main1}
If $M\subseteq L$ is a tight embedding, then, with $V=\frac{\fm
L}{\fm M}$, there exists a coefficient-wise inequality
\[
H_{V}(s)P_k^R(s,t)\preceq P_M^R(s,t)\prod_{i=1}^e(1+s^{d_i}t).
\label{maineq1}
\]
\end{theorem}
\begin{proof}
We set the following notation
\[
\ov M=M/\fm M,\quad \ov L=L/\fm L,
\]
\[
N=L/M,\quad W=L/\fm M.
\]
Consider the following commutative diagram
\begin{equation} \label{DiagCoho}
\xymatrixcolsep{2pc}
\xymatrixrowsep{2pc}
\xymatrix{
0 \ar@{->}[r] & M \ar@{->}[r] \ar@{->}[d] & L \ar@{->}[r] \ar@{->}[d] & N \ar@{->}[r] \ar@{=}[d] &0\\
0 \ar@{->}[r] & \ov{M} \ar@{->}[r]  & W \ar@{->}[r]  & N \ar@{->}[r] &0
}
\end{equation}
This induces a commutative diagram of homomorphisms of bigraded
left $\me$-modules.
\begin{equation} \label{diagrcoho}
\xymatrixcolsep{2pc}
\xymatrixrowsep{2pc}
\xymatrix{
\me(L) \ar@{->}[r] & \me(M) \ar@{->}[r]^{\eth}  & \Sigma\me(N)  \\
\me(W) \ar@{->}[r] \ar@{->}[u]  & \me(\ov{M})  \ar@{->}[u] \ar@{->}[r]^{\delta}  & \Sigma\me(N), \ar@{=}[u]
}
\end{equation}
where $\eth$ and $\delta$ are connecting homomorphisms.

As $\fm$ annihilates $\ov M$ and $N$, there are natural isomorphisms $\me(\ov M)\cong\me\otimes_k\ov M^\vee$ and $\me(N)\cong\me\otimes_k N^\vee$.
Let $\mk$ denote the universal enveloping algebra of $\pi^{\geq2,*}(R)$.

Consider the following commutative diagram
\[
\xymatrixcolsep{3pc}
\xymatrixrowsep{2pc}
\xymatrix{
\me(\ov M) \ar@{->}[r]^{\delta} &  \Sigma\me(N)  \\
\ov M^\vee \ar@{->}[r]^{\delta^0} \ar@{->}[u]  & \me^1\otimes_kN^\vee \ar@{->}[u]
}
\]
where the vertical maps are natural injective maps. Fix the $k$-basis
\[
\mu_1,\ldots,\mu_m\;\;\mathrm{of}\;\;\ov M^\vee,\quad
\nu_1,\ldots,\nu_q\;\;\mathrm{of}\;\;N^\vee,\quad
\xi_1,\ldots,\xi_e\;\;\mathrm{of}\;\;\me^1.
\]
The map $\delta$ is $\me$-linear and so is also $\mk$-linear. Hence, if $\varepsilon\in\me$
\[
\delta(\varepsilon\otimes\mu_h)=\varepsilon\delta^0(\mu_h)=\Sigma_{i,j}a_{i,j,h}\varepsilon\xi_j\otimes
\nu_i
\]
where
\[
\delta^0(\mu_h)=\Sigma_{i,j}a_{i,j,h}\xi_j\otimes
\nu_i\quad\mathrm{with}\;a_{i,j,h}\in k.
\]
By the Poincar\'{e}-Birkhoff-Witt theorem, $\me$ is a free $\mk$-module with basis
\begin{equation} \label{pbw.infreg}
\{\xi_{j_1}\cdots\xi_{j_k}\}_{j_1<\cdots<j_k}.
\end{equation}
Hence, $\me(N)$ is a free $\mk$-module with basis
\begin{equation} \label{pbw.N}
\{\xi_{j_1}\cdots\xi_{j_k}\otimes\nu_i\}_{j_1<\cdots<j_k,i}.
\end{equation}
If $\kappa\in\mk$, then
\[
\delta(\kappa\otimes\mu_h)=\Sigma_{i,j}a_{i,j,h}\kappa\xi_j\otimes
\nu_i.
\]
This means that as a $\mk$-module $\im\delta |_{\mk\otimes_k\ov
M^\vee}$ is generated by the elements $\delta^0(\mu_h)$.

We can change the basis of $\ov M^\vee$ such that the coordinate vectors of
\[
\delta^0(\mu_1),\ldots,\delta^0(\mu_{m'}),
\]
with respect to the basis \eqref{pbw.N}, are linearly independent over $k$ and
\[
\delta^0(\mu_{m'+1}),\ldots,\delta^0(\mu_m)
\]
are all zero. Because the elements $\xi_j\otimes\nu_i$ are part of a $\mk$-basis of $\me(N)$ we deduce that $\delta^0(\mu_1),\ldots,\delta^0(\mu_{m'})$ are linearly independent over $\mk$.

As a result
\[
\im\delta |_{\mk\otimes_k\ov M^\vee}=\mk\otimes_k\im\delta^0.
\]

This means that $\me(N)$ contains a copy of
$\mk\otimes_k\im\delta^0$, and based on the commutativity of the diagram
\eqref{diagrcoho}, this gives

\[
H_{\me(M)}(s,t)\succeq H_{\mk\otimes_k\im\delta^0}(s,t)=H_{\mk}(s,t)H_{\im\delta^0}(s).
\]
According to \eqref{pbw.infreg}, there is an equality of formal power series
\[
H_{\mk}(s,t)=\frac{H_{\me}(s,t)}{H_{\bigwedge \me^1}(s,t)}.
\]
Because $\me^1\cong(\fm/\fm^2)^\vee$ and $H_{X^\vee}(s)=H_X(s^{-1})$ for every $R$-module $X$ annihilated by $\fm$, we deduce
\[
H_{\bigwedge \me^1}(s,t)=\prod_{i=1}^e(1+s^{-d_i}t).
\]
Now, consider the following chain of equalities of Hilbert series
\begin{align*}
H_{\im\delta^0}(s)&=H_{\ov M^\vee}(s)-H_{W^\vee}(s)+H_{N^\vee}(s)\\
&=H_{\ov M}(s^{-1})-H_{W^\vee}(s)+H_N(s^{-1})\\
&=H_{W}(s^{-1})-H_{W^\vee}(s)\\
&=H_{V}(s^{-1})+H_{\ov L}(s^{-1})-H_{W^\vee}(s)\\
&=H_{V}(s^{-1}).
\end{align*}
The first equality follows from the exact sequence in cohomology induced by the second row in \eqref{DiagCoho}, i.e.,
\[
0\rightarrow N^\vee\rightarrow W^\vee\rightarrow \ov{M}^\vee\rightarrow \im\delta^0\rightarrow0,
\]
the second equality follows from the already mentioned fact that $H_{X^\vee}(s)=H_X(s^{-1})$ for every graded $R$-module $X$ annihilated by $\fm$.
\noindent
The third equality follows from the second row in \eqref{DiagCoho}, the fourth equality follows from
\[
0\rightarrow V\rightarrow W\rightarrow \ov L\rightarrow0,
\]
and the last equality follows from $H_{W^\vee}(s)=H_{\ov L}(s^{-1})$, indeed
\begin{align*}
\Hom_R(W,k)&\cong\Hom_R(W,\Hom_k(k,k))\\
&\cong\Hom_k(\ov L,k),
\end{align*}
where the second isomorphism comes from adjunction, therefore $H_{W^\vee}(s)=H_{\ov{L}^\vee}(s)=H_{\ov L}(s^{-1})$.

Holistically we obtain
\begin{align*}
H_{\me(M)}(s,t)&\succeq H_{\mk\otimes_k\im\delta^0}(s,t)\\
&=H_{\mk}(s,t)H_{\im\delta^0}(s,t)\\
&=\frac{H_{\me}(s,t)}{\prod_{i=1}^e(1+s^{-d_i}t)}H_{V}(s^{-1}).
\end{align*}
Now, we observe that
\[
H_{\me(M)}(s,t)=P^R_M(s^{-1},t)\quad\mathrm{and}\quad H_{\me}(s,t)=P^R_k(s^{-1},t).
\]

Hence, we obtain
\[
P^R_M(s^{-1},t)\succeq\frac{P^R_k(s^{-1},t)}{\prod_{i=1}^e(1+s^{-d_i}t)}H_{V}(s^{-1}),
\]
and replacing $s^{-1}$ with $s$, we obtain the desired inequality.
\end{proof}

Let $M$ be a finitely generated graded $R$-module. The \emph{Castelnuovo-Mumford regularity} of $M$ is
\[
\reg M=\sup\{j-i\mid\beta_{i,j}(M)\neq0\}.
\]

\begin{definition}
A ring $R$ is \emph{Koszul} if the algebra $\ext_R(k,k)$ is generated by $\ext_R^{1,\_}(k,k)$.
\end{definition}

According to \cite{AP}, if $\reg k<\infty$, then $\reg k=\Sigma_i (d_i-1)$; this result can happen if and only if $R$ is Koszul.

\begin{lemma}\label{SeriesProd1}
Let
\[
P(s,t)=\sum_{i,j}a_{i,j}s^jt^i\in\mathbb{Z}[s^{\pm1}][[t]],\quad Q(s,t)=\sum_{i,j}b_{i,j}s^jt^i\in\mathbb{Z}[s^{\pm1}][[t]],
\]
with $a_{i,j},b_{i,j}\geq0$ for all $i,j$. Let
\[
P(s,t)Q(s,t)=\sum_{i,j}c_{i,j}s^jt^i.
\]
\begin{enumerate}
\item If $a_{i,j}=0$ for $j-i>r_1$ and $b_{i,j}=0$ for $j-i>r_2$, then $c_{i,j}=0$ for $j-i>r_1+r_2$.
\item If $a_{i,j}=0$ for $j-i>r_1$, $a_{0,r_1}\neq 0$ and $c_{i,j}=0$ for $j-i>r_2$, then $b_{i,j}=0$ for $j-i>r_2-r_1$.
\end{enumerate}
\end{lemma}
\begin{proof}

\begin{enumerate}
\item By the definition of product in the ring $\mathbb{Z}[s^{\pm1}][[t]]$ the coefficients $c_{i,j}$ are given by the formula
\[
c_{i,j}=\sum_{p,q}a_{p,q}b_{i-p,j-q}.
\]
If $q-p>r_1$ then $a_{p,q}b_{i-p,j-q}=0$. If $q-p\leq r_1$ (and $j-i>r_1+r_2$) then $j-q-(i-p)=j-i-(q-p)\geq j-i-r_1>r_1+r_2-r_1=r_2$ and therefore $a_{p,q}b_{i-p,j-q}=0$. This shows that if $j-i>r_1+r_2$ then $c_{i,j}=0$.

\item As before the following equality holds for all $i,j$
\[
c_{i,j+r_1}=\sum_{p,q}a_{p,q}b_{i-p,j+r_1-q}.
\]
If $j-i>r_2-r_1$ then $c_{i,j+r_1}=0$, therefore $a_{p,q}b_{i-p,j+r_1-q}=0$ for all $p,q$. In particular $a_{0,r_1}b_{i,j}=0$ and, since $a_{0,r_1}\neq0$, we deduce that $b_{i,j}=0$.

\end{enumerate}
\end{proof}

\begin{definition}
If $V$ is a finite dimensional graded $k$-vector space then $\mathrm{end}\;V$ is the largest $j$ such that $V_j\neq0$.
\end{definition}

\begin{corollary}
If $M$ is a tightly embeddable module and $\reg M<\infty$, then $R$ is Koszul.
\end{corollary}
\begin{proof}
Set $r:=\reg M$. If $\reg M<\infty$, then $\beta_{i,j}^R(M)=0$ for $j-i>r$. If
\[
\sum_{i,j}a_{i,j}s^jt^i=P_M^R(s,t)\prod_{i=1}^e(1+s^{d_i}t),
\]
then, by Lemma \ref{SeriesProd1}(1), $a_{i,j}=0$ for $j-i>r+\sum_{i=1}^ed_i$. Set a tight embedding $M\subseteq L$ and denote $\fm L/\fm M$ by $V$, by Theorem \ref{maineq1} if
\[\sum_{i,j}c_{i,j}s^jt^i=H_{V}(s)P_k^R(s,t),
\]
then $c_{i,j}=0$ if $j-i>r+\sum_{i=1}^ed_i$. Because $V\neq0$ we deduce, using Lemma \ref{SeriesProd1}(2), that $\beta_{i,j}^R(k)=0$ if $j-i>r-\mathrm{end}\;V+\sum_{i=1}^ed_i$, which is equivalent to $\reg k<\infty$.
\end{proof}

\section{Direct summands of syzygies of the residue field}
In this section, we prove that nonzero direct summands of syzygies of $k$ have infinite regularity if $R$ is a complete intersection that is not Koszul. Special homological properties of this class of modules were already noticed in \cite{extremal,Mart,ryo}.
\begin{remark}
We point out that if $R$ is a complete intersection then $R$ is Koszul if and only if $I$ is generated by quadrics.
\end{remark}
\begin{theorem}\label{main2}
If $R$ is a complete intersection, $M=\Omega^mk$ with $m\geq1$, $\beta\colon M\rightarrow N$ is an $R$-module homomorphism of finitely generated graded $R$-modules such that for some $n$, the map induced by $\beta$, $\beta^n\colon\ext_R^n(N,k)\rightarrow\ext_R^n(M,k)$ is not zero, then for some $b\in\mathbb{Z}$, there is a coefficient-wise inequality
\[
s^bP_k^R(s,t)\preceq P^R_{\Omega^n N}(s,t)\prod_{i=1}^e(1+s^{d_i}t).
\]
\end{theorem}
\begin{proof}
We set the following notation
\[
N^\prime=\Omega^nN, \quad \ov N^\prime=N^\prime/\fm N^\prime, \quad M^\prime=\Omega^{m+n}k.
\]
Let $\pi$ be the canonical projection $N^\prime\rightarrow N^\prime/\fm N^\prime$ and $\beta^\prime$ be a morphism $M^\prime\rightarrow N^\prime$ obtained by extending $\beta$ to a morphism of free resolutions.

Let $\xi$ denote the composed map
\[
\xi^*:\me\otimes_k(\ov N^\prime)^\vee\cong\me(\ov
N^\prime)\xra{\pi^*}\me(N^\prime)\xra{\beta^{\prime
*}}\me(M^\prime)\xra{\alpha^*}\Sigma^{m+n}\me^{\geq m+n,\_},
\]
where $\alpha^*$ is an iterated connecting homomorphism, and hence is an isomorphism. By construction, $\xi^0\neq0$ since $\pi^0,\alpha^0\neq0$ and $(\beta')^0=\beta^n\neq0$.

As a result of the Poincar\'{e}-Birkhoff-Witt Theorem, $\me\cong\mk\otimes_k\bigwedge\me^1$ as (bi-)graded left $\mk$-modules, if $R$ is a complete intersection, then $\pi^{i,\_}(R)=0$ if $i\geq3$ (see for example, \cite[Theorem 7.3.3 and Theorem 10.1.2]{infinite} and the references given there). Because $\pi^{i,\_}(R)=0$ if $i\geq3$, $\mk$ is a polynomial ring. Hence, $\im\xi^*$ is a torsion-free $\mk$-module (because $\me^{\geq m+n,\_}$ is a submodule of a free module over the polynomial ring $\mk$). As a result, any nonzero element in the image of $\xi^0$ generates a copy of $\mk$ in $\im\xi^*$, whose internal degree might be shifted by some $b\in\mathbb{Z}$. It follows that
\[
H_{\me(N^\prime)}(s,t)\succeq s^b
H_{\mk}(s,t)=s^b\frac{H_{\me}(s,t)}{\prod_{i=1}^e(1+s^{-d_i}t)}.
\]
Now, we conclude as in Theorem \ref{main1}.
\end{proof}

\begin{corollary}
Let $R$ be a complete intersection with $\reg k=\infty$. For any $n\geq0$, each nonzero direct summand of $\Omega^nk$ has infinite regularity.
\end{corollary}
We raise the following question: can the complete intersection hypothesis be removed from Theorem \ref{main2}? More specifically:
\begin{question}\label{question}
If $\beta\colon \Omega^mk\rightarrow N$ is an $R$-module homomorphism of finitely generated graded $R$-modules such that for some $n$ the map
\[
\beta^n\colon\ext_R^n(N,k)\rightarrow\ext_R^n(\Omega^mk,k)
\]
is not zero and $N$ has finite regularity, then is $R$ Koszul?
\end{question}

\section{On the vanishing of the graded deviations}
In this section, we relate Question \ref{question} and the vanishing
of the graded deviations of $R$. Let $R\langle X\rangle$ be an
acyclic closure of $k$ (see \cite[6.3]{infinite} for the
definition). Over graded rings, we require the differential of the
acyclic closure to be a homogeneous map. For the
differential to be homogeneous, we must give the elements of $X$ an
internal grading, making $X$ a bigraded set.
\begin{definition}
The $(i,j)$th graded deviation of $R$ is
\[
\varepsilon_{i,j}(R):=\mathrm{Card}(X_{i,j}),
\]
where $X_{i,j}$ is the set of variables in the acyclic closure of homological degree $i$ and internal degree $j$.
\end{definition}

\begin{theorem}
If $M=\Omega^mk$ for $m\geq1$ and $\beta\colon M\rightarrow N$ is a homomorphism of finitely generated graded $R$-modules such that for some $n$, the map induced by $\beta$,
\[
\beta^n\colon\ext_R^n(N,k)\rightarrow\ext_R^n(M,k)
\]
is not zero and $\reg N<\infty$, then $\varepsilon_{i,j}(R)=0$ for $i>m+n$ and $i\neq j$.
\end{theorem}
\begin{proof}
We set the following notation
\[
N^\prime=\Omega^nN, \quad \ov N^\prime=N^\prime/\fm N^\prime, \quad M^\prime=\Omega^{m+n}k.
\]
Let $\pi$ be the canonical projection $N^\prime\rightarrow N^\prime/\fm N^\prime$ and $\beta^\prime$ be a morphism $M^\prime\rightarrow N^\prime$ obtained by extending $\beta$ to a morphism of free resolutions.
Let $\mv$ be the universal enveloping algebra of $\pi^{>m+n,\_}(R)$.
Let $\xi^*$ denote the composed map
\[
\xi^*:\me(\ov N^\prime)\xra{\pi^*}\me(N^\prime)\xra{\beta^{\prime
*}}\me(M^\prime)\xra{\alpha^*}\Sigma^{m+n}\me^{\geq m+n,\_}.
\]

Consider the following commutative diagram
\[
\xymatrixcolsep{6.5pc}
\xymatrixrowsep{3pc}
\xymatrix{
\me(\ov N^\prime) \ar@{->}[r]^{\xi^*}&\Sigma^{m+n}\me^{\geq m+n,\_}\\
(\ov N^\prime)^\vee\ar@{->}[u]\ar@{->}[r]^{\xi^0}&\me^{m+n,\_}\ar@{->}[u]
}
\]

Let $\theta_1,\ldots,\theta_p$ be a $k$-basis of $\pi^{<m+n,\_}(R)$ with $|\theta_1|\leq\cdots\leq|\theta_p|$. By the Poincar\'{e}-Birkhoff-Witt theorem, a $\mv$-basis of
$\me^{\geq m+n,\_}$ is given by

\begin{equation} \label{pbw.2}
\{\theta_1^{(i_1)}\cdots\theta_p^{(i_p)}\mid i_j\leq
1\mathrm{\;if\;}|\theta_j|\mathrm{\;odd\;},
\sum_ji_j|\theta_j|\geq m+n\}.
\end{equation}

The $\mv$-module $\im\xi^* |_{\mv\otimes_k(\ov N')^\vee}$ is generated
by the elements $\xi^0(\nu_h)$, where
\[
\nu_1,\ldots,\nu_q
\]
is a $k$-basis of $(\ov N')^\vee$. We set
\[
\xi^0(\nu_h)=\sum
a_{i_1,\ldots,i_p}\theta_1^{(i_1)}\cdots\theta_p^{(i_p)},\quad
a_{i_1,\ldots,i_p}\in k.
\]
We can change the basis of $(\ov N')^\vee$ so that the coordinate vectors of
\[
\xi^0(\nu_1),\ldots,\xi^0(\nu_{q'}),
\]
with respect to the basis \eqref{pbw.2}, are linearly independent over $k$ and
\[
\xi^0(\nu_{q'+1}),\ldots,\xi^0(\nu_{q})
\]
are all zero. Because the elements $\theta_1^{(i_1)}\cdots\theta_p^{(i_p)}$ form a basis of $\me^{\geq  m+n,\_}$ over $\mv$, we deduce that $\xi^0(\nu_1),\ldots,\xi^0(\nu_{q'})$ are also linearly independent over $\mv$.

Therefore
\[
\im\xi^* |_{\mv\otimes_k(\ov N')^\vee}=\mv\otimes\im\xi^0.
\]

This means that $\im\xi^*$ contains a copy of $\mv$ and by the construction of $\xi^*$ so does $\me(N')$.

Now, we recall that by \cite[Theorem 10.2.1]{infinite} $\dim_k\pi^{i,j}(R)=\varepsilon_{i,j}(R)$.

If there is an even $i>m+n$ such that $\varepsilon_{i,j}(R)\neq0$
and $i\neq j$, then there is a nonzero element $x\in\pi^{i,j}(R)$.
The powers of $x$ belong to $\mv$. However, a copy of $\mv$ is contained
in $\me(N')$. Because $\mathrm{bideg}\;x=(i,j)$, then
$\mathrm{bideg}\;x^{(l)}=(li,lj)$ and $lj-li=l(j-i)$ goes to
$\infty$ as $l$ goes to $\infty$. This implies that
$\reg N^\prime=\infty$, which is a contradiction.

If there are infinitely many $i$ with $i>m+n$, $i$ odd, $i\neq0$,
and $\varepsilon_{i,j}(R)\neq0$, then there are infinitely many
nonzero $x_t$ with $t=1,2,\ldots$ belonging to $\pi^{i_t,j_t}(R)$.
The products $x_1x_2\cdots x_s$ belong to $\mv$. However, a copy of $\mv$
is contained in $\me(N')$. The bidegree of this product
is
\[
\mathrm{bideg}\;x_1x_2\cdots x_s=(i_1+i_2+\cdots+i_s,j_1+j_2+\cdots+j_s),
\]
and because $j_t-i_t\geq1$ for every $t$, the following inequality holds
\[
j_1+j_2+\cdots+j_s-(i_1+i_2+\cdots+i_s)\geq s,
\]
and as $s$ goes to $\infty$, we obtain $\reg N^\prime=\infty$, which is a contradiction.

Now, we assume that $\varepsilon_{i,j}(R)\neq0$ for finitely many
$i$ with $i>m+n$, $i$ odd, and $i\neq j$. Let $R\langle X\rangle$
be the acyclic closure of $k$. We
denote by $|.|$ the homological degree of a homogeneous element in
the acyclic closure and by $\deg$ its internal degree and we say that a homogeneous element $x$ is \emph{off-diagonal} if $|x|\neq\deg\;x$. We can choose a homogeneous element $y\in X$
that is of odd degree, off-diagonal, and of maximal homological degree. We claim that
$A=R\langle X\backslash\{y\}\rangle$ is a DG subalgebra of the
acyclic closure. Indeed, if $\tilde{x}_1,\ldots,\tilde{x}_r,y$ are all the odd
elements and off-diagonal, then, by \cite[Lemma 7]{AP}, $R\langle
X\backslash\{\tilde{x}_1,\ldots,\tilde{x}_r,y\}\rangle$ is a DG subalgebra of the
acyclic closure. To prove that $A$ is a DG subalgebra it suffices to prove that if $x\in X\backslash\{y\}$, then the differential of $x$ belongs to $A$. If $x\neq \tilde{x}_1,\ldots,\tilde{x}_r$ then this is true because $R\langle
X\backslash\{\tilde{x}_1,\ldots,\tilde{x}_r,y\}\rangle$ is a DG algebra, if $x=\tilde{x}_t$ for some $t=1,\ldots,r$ then the variable $y$ cannot appear in the differential of $x$ since, by maximality, $|y|\geq|x|$, therefore in all cases the differential of $x$ belongs to $A$ and hence $A$ is a DG subalgebra of the acyclic closure. Since $A$ is a DG algebra, the acyclic closure
of $k$ can be written as $A\langle y\rangle$, implying (by
\cite[Lemma 6]{AP}) that $|y|=1$, which is a contradiction because
$m\geq1$.
\end{proof}

\begin{proposition}
Let $d\geq4$ be an even number. If $\varepsilon_{i,j}(R)=0$ for $i\geq d$ and $i\neq j$, then $\varepsilon_{i,j}(R)=0$ for $i\geq d-1$ and $i\neq j$.
\end{proposition}
\begin{proof}
Let $R\langle X\rangle$ be an acyclic closure of $k$, and let $\partial$ be its differential. Let $y$ be an element of $X$ of homological degree $d-1$, we need to prove that $\deg y=d-1$. We assume that $y$ appears in a boundary of a (bi)homogeneous element $x$ with $|x|\geq d$
\[
\partial x=a+rwy,\quad a,w\in R\langle X\backslash\{y\}\rangle, r\in\fm,
\]
where $R\langle X\backslash\{y\}\rangle$ is just a subalgebra of the acyclic closure and not a DG subalgebra, and $r$ is in $\fm$ because of the minimality of the acyclic closure.

The following string of (in)equalities is observed
\begin{align*}
\deg w+\deg y &=\deg(wy)\\
              &\geq|wy|\\
              &= |x|-1=\deg x-1\\
              &=\deg w+\deg y+\deg r-1\\
              &\geq\deg w+\deg y,
\end{align*}
where $\deg r\geq1$ since $r\in\fm$ and $d_i\geq1$ for all $i=1,\ldots,e$.

We deduce that
\[
\deg w+\deg y=|w|+|y|.
\]
If $\deg y>|y|$, then $\deg w<|w|$, which is not possible; hence, $\deg y=|y|$.

If $y$ does not appear in the boundary of any element, then we can write the acyclic closure as $A\langle y\rangle$ with $A=R\langle X\backslash\{y\}\rangle$, and by \cite[Lemma 6]{AP} $|y|=1$,  hence $d=2$, which is not possible.
\end{proof}

We set $\varepsilon_i(R)=\Sigma_j\varepsilon_{i,j}(R)$. It is known that if $\varepsilon_i(R)=0$ for $i\gg0$, then $\varepsilon_i(R)=0$ for $i\geq3$, see \cite[Theorem 7.3.3]{infinite}. Motivated by this and by the previous proposition, we raise the following question
\begin{question}
If $\varepsilon_{i,j}(R)=0$ when $i\gg0$ and $i\neq j$, is it true that $\varepsilon_{i,j}(R)=0$ when $i\geq3$ and $i\neq j$?
\end{question}

\noindent
By \cite[Theorem 7.3.3]{infinite} and \cite[Theorem 7.3.2]{infinite},

1) $\varepsilon_i(R)=0$ for $i\geq 2$ if and only if $R$ is regular, and

2) $\varepsilon_i(R)=0$ for $i\geq 3$ if and only if $R$ is a complete intersection.

\noindent
By \cite[Theorem 2]{AP}, if $\varepsilon_{i,j}(R)=0$ when $i\geq2$ and $i\neq j$, then the ring $R$ is of the form $Q\otimes_kS$ with $Q$ regular and $S$ a standard graded Koszul ring. Motivated by these results, we raise the following question
\begin{question}
If $\varepsilon_{i,j}(R)=0$ when $i\geq3$ and $i\neq j$, then is the ring $R$ of the form $Q\otimes_kS$ with $Q$ a complete intersection and $S$ a standard graded Koszul ring?
\end{question}

\section*{Acknowledgements}
The author would like to thank his advisors L. Avramov and S. Iyengar for helpful conversations.


\begin{thebibliography}{99}

\bibitem{andre}
M. Andr\'{e},
\textit{Hopf algebras with divided powers},
J. Algebra \textbf{18} 1971 19-–50.

\bibitem{infinite}
L.~L.~Avramov,
\textit{Infinite free resolutions},
Six Lectures on Commutative Algebra (Bellaterra, 1996), Progr. Math. \textbf{166}, Birkh\"{a}user, Basel, 1998; pp. 1--118.

\bibitem{extremal}
L. L. Avramov \textit{Modules with extremal resolutions}, Math. Res.
Lett. \textbf{3} (1996), 319-–328.


\bibitem{ae}
L. L. Avramov, D. Eisenbud,
\textit{Regularity of modules over a Koszul algebra},
J. Algebra \textbf{153} (1992), no. 1, 85-–90.

\bibitem{AP}
L. L. Avramov, I. Peeva,
\textit{Finite regularity and Koszul algebras},
Amer. J. Math. \textbf{123} (2001), 275--281

\bibitem{Levin}
T. H. Gulliksen, G. Levin,
\textit{Homology of local rings},
Queen's Papers Pure Appl. Math. 20, Queen's Univ., Kingston, ON (1969).

\bibitem{Mart}
A. Martsinkovsky,
\textit{A remarkable property of the (co) syzygy modules of the residue field of a nonregular local ring},
J. Pure Appl. Algebra \textbf{110} (1996), no. 1, 9-–13.

\bibitem{MilnorMoore}
J. W. Milnor, J. C. Moore,
\textit{On the structure of Hopf algebras},
Ann. of Math. (2) \textbf{81} 1965 211–-264.



\bibitem{sjodin}
G. Sj\"{o}din,
\textit{Hopf algebras and derivations},
J. Algebra \textbf{64} (1980), no. 1, 218-–229.

\bibitem{ryo}
R. Takahashi,
\textit{Syzygy modules with semidualizing or G-projective summands},
J. Algebra \textbf{295} (2006), no. 1, 179-–194.

\end{thebibliography}
\end{document}